\documentclass[titlepage,draft,11pt]{article} 
\usepackage{amsfonts,latexsym} 
\usepackage{pstricks,pst-plot}
\usepackage{amsmath} 
\usepackage{amsthm}
\usepackage[a4paper]{geometry}

\date{}

\newcommand{\re}{\mathbb{R}}

\numberwithin{equation}{section}

\newtheorem{thm}{Theorem}[section]

\newtheorem{rmk}[thm]{Remark}
\newtheorem{prop}[thm]{Proposition}

\newtheorem{cor}[thm]{Corollary}

\title{ A Newtonian approach to general black holes}

\author{Alain Haraux\vspace{1ex}\\ 
{\normalsize Sorbonne Universit\'e, Universit\'e Paris-Diderot SPC, CNRS, INRIA}, \\
{\normalsize Laboratoire Jacques-Louis Lions,  LJLL, F-75005,
Paris, France.}\\ 
{\normalsize e-mail: \texttt{alain.haraux@sorbonne-universite.fr}}}

\begin{document}
\maketitle
\begin{abstract}
We show how it is possible to define and study black holes of arbitrary shapes in the framework of Newtonian mechanics.  \\

\vspace{1cm} 

 \end{abstract}

%%%%%%%%%%%%%%%%%%%%%
%                   %
%   Inizio lavoro   %
%                   %
%%%%%%%%%%%%%%%%%%%%%
 
\section{Introduction} According to Wikipedia, objects whose gravitational fields are too strong for light to escape were first considered in the 18th century by the English astronomical pioneer and clergyman John Michell and  a few years later by Pierre-Simon de Laplace. The idea was briefly proposed by Michell in a letter published in November 1784 \cite{Mic}.  It was a natural idea within the framework of Newton's gravitational theory, assuming that light was produced by corpuscles of positive mass, as assumed by Newton. Our present physics models discredit Michell's notion of a light ray shooting directly from the surface of a supermassive star, being slowed down by the star's gravity, stopping, and then free-falling back to the star's surface. Actually the general relativity framework assumes a constant velocity for light and more generally all electro-magnetic waves travelling in the empty space. \\

The concept reappeared in a slightly different form after the creation of general relativity by Albert Einstein.  In 1916, Karl Schwarzschild \cite{Sch} found the first modern relativistic solution that would characterize a black hole. Then David Finkelstein, in 1958, first published the interpretation of ``black hole" as a region of space from which nothing can escape. It was not until the 1960s that theoretical work showed they were a generic prediction of general relativity, but nowadays almost all texts on black holes mention them only in connection with the general relativity model.  \\

In the recent years, the black hole concept has attracted more and more attention since it is conjectured that a giant black hole is sitting in the center of every galaxy, including our milky way.  There are polemics about the existence of such objects in the exact form proposed by Karl Schwarzschild, and indeed it is not clear whether the fine structure of atomic particles really allow the formation of a point singularity, and moreover the effect of gravity at very small distances is not known. It might even become repulsive as conjectured in the framework of some alternative gravitational models. \\

In the present paper we shall not assume anything for small distances and our main point will be to understand the exact kind of lockdown undergone by both matter and light in the vicinity of a very large mass concentration. Since we wish to understand what happens for arbitrary, not necessarily spherical matter concentrations, we adopt the Newtonian formulation of gravity. As a drawback, we need to drop the assumption of photons travelling exactly at the speed of light. We shall justify this point of view in our conclusion. \\

The paper is organized as follows: Section 2 is devoted to the classical spherically symmetric case. It is rather natural to consider first the spherically symmetric space because the study of black holes started with the idea of "black stars" and stars are usually close to spherical with a spherically symmetric distribution of mass. Section 3 concerns the case of arbitrary distributions which might turn out to be relevant in the future if we consider low density extremely large lack holes for which symmetrization under gravity is not yet achieved. In Section 4, we apply the results of Sections 2 and 3 to the case of black holes. Section 5 is devoted to a few remarks about cosmology.

  \section{The escape velocity in the spherically symmetric case}  \subsection{The context} In this section, we recall the principle on which John Michell based his proposal on existence of ``black stars". Let us consider a spherical body of radius $r>0$ with local density depending only on the distance to its center, taken as the origin of coordinates.  It has been known for a long time, at least since Gauss,  that outside the sphere of radius $r$, the distributed mass produces the same gravitational field as if the total mass was located at the center.  Therefore, in the absence of other forces, the equation of motion of a small body, modelized by a point mass located at the spatial position $u\in\re^3$ is \begin{equation} \label{grav.S} u'' = - \frac {GM u} {||u||^2} \end{equation} where $M$ is the total mass of the body and $G$ the gravitational constant. The equation does not depend on the mass of the small body, which could, in the classical framework, be a photon. Associated with \eqref{grav.S}, we have a conserved energy  \begin{equation} \label{en.S} E(u, u') = \frac{1}{2}||u'||^2- \frac {GM} {||u||} \end{equation} From equation \eqref{grav.S}, it follows that starting from any initial state $ u_0, v_0 $  with $||u_0||> r$, there is a unique local solution $ u \in C^2([0, T_{max}), \re^3)$ of \eqref{grav.S} such that $ u(0) = u_0,  u'(0) = v_0 .$ In addition, the only way for the solution $u$ to have a finite time of existence is collapsing to the origin. For this solution to be a true solution of the mechanical problem, we need in addition to warrant $ ||u(t)||>r, \quad \forall t \in [0, T_{max}) $, since the equation ceases to be valid when we enter the ball of radius $r$. By taking the inner product  of both sides of equation \eqref{grav.S}, it is immediate to check that the energy $E$ is conserved: 
  \begin{equation} \label{en-cons} \forall t \in [0, T_{max}), \quad E(u(t), u'(t)) = E(u(0), u'(0))\end{equation} Due to the structure of \eqref{en.S}, it is now clear that if the trajectory $u(t)$ passes through arbitrarily distant points, we must have  $$ E= E(u(0), u'(0))\ge 0. $$ It turns out that as soon as this condition is violated, the trajectory $u(t)$ remains bounded, and the condition for boundedness is optimal. 
  \subsection{Below the escape velocity.} It turns out that as soon as $E<0$, the trajectory cannot escape. We now consider, for some $R>r$,  the sphere $$S_R = \{ u\in\re^3, \quad ||u||= R\} $$  Our main result is 
  \begin{thm}\label{conf} Let $u_0\in S_R$ and $ \displaystyle ||v_0||^2 < \frac{2GM}{R}. $  Then \begin{equation} \label{inconf} \forall t \in [0, T_{max}), \quad ||u(t)|| \le \frac{R}{1- \frac{R||v_0||^2}{2GM}}\end{equation} \end{thm} \begin{proof} By the energy conservation property, we have $$ \forall t \in [0, T_{max}), \quad \frac{1}{2}||u'(t)||^2- \frac {GM} {||u(t)||} = \quad \frac{1}{2}||v_0||^2- \frac {GM} {R}: =-\eta $$ Hence 
  $$ \forall t \in [0, T_{max}), \quad  \frac {GM} {||u(t)||} = \quad \frac{1}{2}||u'(t)||^2+\eta\ge \eta =  \frac {GM} {R} - \frac{1}{2}||v_0||^2 $$  The resulting inequality  
  $$ \forall t \in [0, T_{max}), \quad ||u(t)|| \le \frac{GM}{\eta} = \frac{GM}{\frac {GM} {R} - \frac{1}{2}||v_0||^2}  $$ is equivalent to \eqref{inconf}
  \end{proof} 
  \subsection{Two results of optimality.} 
  The result of Theorem \ref{conf}  is optimal in two directions. 
  \begin{prop} If $u_0\in S_R$ and $ \displaystyle ||v_0||^2 = \frac{2GM}{R}$, in some cases the solution $u$ grows up, which means that $T_{max} = \infty $ and 
  $$ \lim_{t\to\infty} ||u(t)||= \infty. $$ \end{prop} \begin{proof} The condition $ ||v_0||^2 = \frac{2GM}{R} $ means  $E= 0$, in which case we have $$ \forall t \in [0, T_{max}), \quad \frac{1}{2}||u'(t)||^2- \frac {GM} {||u(t)||} = 0 .$$ Let us consider the case where $ u_0, v_0 $ are colinear in which case the solution and its velocity both remain on a straight line passing through $0$. By setting $u(t) = y(t) w$ with $w$ some unit vector, the equation reduces to $$ y' = \pm \left(\frac{\gamma}{|y|} \right)^{1/2} $$ with $\gamma:= 2GM$. Looking for positive solutions we end up with two curves 
  $$ y(t) =  \left(R^{3/2} \pm \frac{3}{2} \sqrt\gamma t \right)^{2/3}  $$ The solution \begin{equation} y_{+} (t) =  \left(R^{3/2} + \frac{3}{2} \sqrt\gamma t \right)^{2/3}  $$ is growing up, while the solution $$ y_{-} (t) =  \left(R^{3/2} -\frac{3}{2} \sqrt\gamma t \right)^{2/3}  \end{equation}  collapses with the spherical body in finite time $$ T_{max} = \frac{2}{3} \frac{R^{3/2}- r^{3/2}}{\sqrt\gamma}$$
  \end{proof} 
  \begin{prop} Under the hypothesis of Theorem \ref{conf} the value $\displaystyle \frac{R}{1- \frac{R||v_0||^2}{2GM}} \,\,\left(= \frac{GM }{\eta}\right) $ is achieved as a maximum of $||u(t)||$ for some particular solutions $u.$ \end{prop} \begin{proof} Setting $\gamma:= 2GM$ and looking for scalar solutions $u = zw$ with $||w||= 1$ as previously, we end up with the ODEs 
  $$ z' = \pm \left(\frac{\gamma}{|z|}-2\eta \right)^{1/2} $$ The solution $z>0$ of \begin{equation}  z' =  \left(\frac{\gamma}{z}-2\eta \right)^{1/2};\quad z(0)= R \end{equation}  corresponds to a solution $u = zw$ of \eqref{grav.S} with $$||u'(0)||^2 = \frac{\gamma}{R}-2\eta = \frac{2GM}{R}-2\eta< \frac{2GM}{R}$$ and grows until it reaches the value $$ z =  \frac{\gamma}{2\eta} =\frac{GM }{\eta}.$$
  \end{proof} 
  
  \section{A more general case } In this section, we consider a bounded open domain $\Omega$ of $\re^3$, and a mass $M$ filling the domain with local density $\mu \in L^1(\Omega).$ Therefore $ \displaystyle M = \int_{\Omega}\mu(x) dx.$  Equation \eqref{grav.S} is replaced by \begin{equation} 
  \label{grav} u'' = - \int_\Omega \frac {G\mu(x)(u-x)} {||u-x||^2} dx \end{equation} valid for solutions with positions outside the compact set $\overline{\Omega}$, in a region where there is no other matter and where the influence of other gravitational sources is negligible. Now the energy becomes 
  \begin{equation} \label{en.} E(u, u') = \frac{1}{2}||u'||^2-\int_\Omega \frac {G\mu(x)} {||u-x||} dx  \end{equation} We observe that whenever $u\not\in \overline{\Omega}$, we have 
  \begin{equation} \label{in} \frac{GM} {\max_{x\in\overline{\Omega}}||u-x||}\le \int_\Omega \frac {G\mu(x)} {||u-x||} dx \le \frac{GM} {dist(u,\overline{\Omega})} \end{equation}  The construction of local solutions when $u_0\not\in \overline{\Omega}$ is routine. We now have

 \begin{thm}\label{bound} Let $u_0 \not\in \overline{\Omega}$ and $v_0\in \re^3$ be such that \begin{equation} ||v_0||^2 < 2\int_\Omega \frac {G\mu(x)} {||u_0-x||} dx. \end{equation}   
 Then $u$ is bounded on $[0, T_{max})$ and more precisely 
 \begin{equation} \label{inconf2} \forall t \in [0, T_{max}), \quad dist(u(t),\overline{\Omega})\le \frac{2GM}{\int_\Omega \frac {2G\mu(x)} {||u_0-x||} dx - ||v_0||^2.}\end{equation} \end{thm}

 \begin{proof} By the energy conservation property the same argument as in the symmetric case now gives, introducing $ \displaystyle \eta = \int_\Omega \frac {G\mu(x)} {||u_0-x||} dx - \frac{1}{2}||v_0||^2 $ , the inequality 
  $$ \forall t \in [0, T_{max}), \quad  \frac{GM} {dist(u,\overline{\Omega})} \ge \int_\Omega \frac {G\mu(x)} {||u-x||} dx =  \frac{1}{2}||u'(t)||^2+\eta\ge \eta . $$  The conclusion follows immediately.  \end{proof}    
  \begin{cor} \label{evas}  Let $u_0 \not\in \overline{\Omega}$ and $v_0\in \re^3$ be such that $$ {\max_{x\in\overline{\Omega}}||u_0-x||}\times ||v_0||^2 < 2GM. $$  Then 
  $$ \forall t \in [0, T_{max}), \quad dist(u(t),\overline{\Omega})\le \frac{2GM}{\int_\Omega \frac {2G\mu(x)} {||u_0-x||} dx - ||v_0||^2} \le \frac{\max_{x\in\overline{\Omega}}||u_0-x||} {1- \frac{||v_0||^2 \max_{x\in\overline{\Omega}}||u_0-x||} {2GM}}.$$ \end{cor} 
  
 \begin{rmk} Corollary \ref{evas} gives an estimate from below of the escape velocity $V$ at some point  $u_0 \not\in \overline{\Omega}$,  given in the general case by the inequality  $$  V^2 \ge\frac{2GM}{\max_{x\in\overline{\Omega}} ||u_0-x||} . $$ By introducing the average density {$\displaystyle d: = \frac{M}{|\Omega|} $} it can be written in the form 
$$  V^2 \ge |\Omega|\frac{2Gd}{\max_{x\in\overline{\Omega}} ||u_0-x||} . $$\end{rmk}

\section{Application to black holes }  A black hole is a massive body for which the escape velocity overpasses the velocity of light $c$. 

\subsection{General results}   \begin{cor} Assume  that there is no gravitational mass except those in $\Omega.$ Let $B$ be a closed ball containing $\overline{\Omega}$, such that  $$c^2\max_{u\in B} \{ \max_{\overline{\Omega}} ||u-x||\}< 2GM. $$  Then $B$ is a black hole. More precisely, if a photon enters the region B at some time $t_0$, the localization $u(t)$ of the photon satisfies  $$ \forall t \in [t_0, T_{max}), \quad dist(u(t),\overline{\Omega})\le \frac{\max_{\overline{\Omega}} ||u(t_0)-x||} {1- \frac{c^2 \max_{\overline{\Omega}} ||u(t_0)-x||}{2GM}}.$$\end{cor} 

\begin{rmk} After Tmax, several things may happen: the photon may disappear in $\overline{\Omega}$ for ever, in which case the inequality remains trivially true but no equation can be written, especially since the photon may be ``digested" by the matter inside $\overline{\Omega}$; it may remain some time in $\overline{\Omega}$ and try to get out again, in which case the corollary can again be applied. Summarizing all the possible cases, we can say that as long as the photon does not disappear, the inequality remains true. So the lockdown is not only local, but global, but this cannot be said directly at the level of the equation which may cease to be satisfied on some time intervals. \end{rmk}

\begin{rmk} The hypothesis that there is no gravitational mass except those in $\Omega$ is of course never satisfied in reality. But the mathematical theorem must be stated under this condition. In practice physicists will consider that if the other masses are far enough, the lock-down property will be satisfied with almost the same inequality. This is a stability assumption which looks reasonable, but is not at all easy either to formulate or to prove mathematically. Without approximations of this type, nothing can be calculated in the real world.  \end{rmk}  

\begin{rmk} \label {inc.}Equation \eqref{grav.S} shows that if a photon is emitted orthogonally from the surface $S_R = \{ u\in\re^3, \quad ||u||= R\} $, its velocity is decreasing at the beginning of a motion, so we see that the Newtonian framework  is inconsistant with the motion of the photon at constant velocity $c$, as already pointed out in the introduction. \end{rmk}  

The previous corollay implies the following simpler looking statement:

\begin{cor} \label {BH-AS} Assume  that there is no gravitational mass except those in $\Omega.$ Let $B$ be a closed ball containing $\overline{\Omega}$, such that  \begin{equation}c^2 diam(B)< 2GM. \end{equation}  Then $B$ is a black hole. More precisely, if a photon enters the region B at some time $t_0$, either the photon disappears in finite time, or its localization $u(t)$ satisfies \begin{equation} \forall t \ge t_0, \quad dist(u(t),\overline{\Omega})\le \frac{diam(B)} {1- \frac{c^2 diam(B)}{2GM}}.\end{equation}\end{cor} 

\begin{rmk} In that last corollary, in the hypothesis we only lose a factor 2 with respect to the spherically symmetric situation. The conclusion is of course weaker, depending on the shape.   \end{rmk}

\subsection{Characterization in terms of average density. } 

In the framework of average densities, Corollary \ref {BH-AS} can be rewritten in the form 
\begin{prop} \label {BH-AS-d} Assume  that there is no gravitational mass except those in $\Omega.$ Let $B$ be a closed ball containing $\overline{\Omega}$, such that  
\begin{equation} d> \frac{c^2 diam(B)} {2G|\Omega|}. \end{equation}  Then $B$ is a black hole. \end{prop} The spherically symmetric case is especially important.

\begin{prop} In the spherically symmetric case, when $\Omega = B(x_0, r)$ and $ B= \overline{\Omega}$ , the sufficient condition following from Theorem \ref {conf} for $B$ to be a black hole is   
\begin{equation}  \label {d-sym} d> \frac{c^2 r} {2G|\Omega|} = \frac{3}{8\pi}\frac{c^2} {G r^2}.\end{equation} \end{prop}

\begin{rmk} If  $\Omega = B(x_0, r)$ and $ B= \overline{\Omega}$, but the mass distribution in $\omega$ is not spherically symmetric, our sufficient condition  becomes  \begin{equation}  \label {d-asym} d> \frac{c^2 r} {G|\Omega|} = \frac{3}{4\pi}\frac{c^2} {G r^2}=:K/r^2. \end{equation}   The two conditions \eqref{d-sym} and \eqref{d-asym} only differ by a factor 2, and we do not know whether the second condition is optimal in general. \end{rmk}

\subsection{Some very large black holes }  An immediate observation is that formula \eqref {d-asym} allows arbitrary low average densities when $r$ tends to infinity. In \cite{St}, the author suggested that the entire observable universe may lie in the interior of a black hole. His calculations were done in the relativistic framework. We can recover this result in the newtonian framework as follows: first we can check easily that $$ K = \frac{3}{4\pi}\frac{c^2} {G}\sim 3.22. 10^{26} kg/m.$$

The radius of the observable universe is presently estimated to be  $$r \sim 4. 10^{26} m$$ It follows that $$ K/r^2 \sim (3,22/16) 10^{-26} kg/m^3$$  approximately  $2.10^{-27} kg/m^3$. \medskip

\noindent Now the average density of the observable universe is presently estimated around $10^{-23} kg/m^3, $
 some 5 000 times more. So the necessary condition for a black hole is fullfilled. Then of course everything will depend on what is outside the observable universe...

\section{Concluding remarks}   In this paper, we presented a classical approach to black holes which is completely independent of the Schwarzschild singular model (cf. \cite{Sch}) and does not make any hypothesis on the distribution of mass inside the black hole. This approach suggests several remarks. 

\begin{rmk} It follows from our last calculation that there might exist, inside our universe, gigantic black holes of about one billion light years diameter having an average density comparable to that of the observable universe. Actually, mass concentrations will usually have a much larger average density, so we can very well imagine that some huge concentrations of baryonic matter are invisible for us. They should manifest themselves by gravitational lens phenomena  and could have been confused with dark matter. Moreover, if the usual optical aberrations which  are found around ``classical'' black holes are probably absent in low density gigantic black holes, they could contribute in a more discreet manner to the ``ghost light" illuminating the cosmos which is not completely explained today {\rm (cf. \cite{Montes-T}).} Finally, if a large proportion of stars were hidden inside such gigantic black holes, this could help to solve Olber's paradox {\rm (cf. \cite{Har, W})} without invoking the hypothesis of a spatially limited universe. \end{rmk}

\begin{rmk} There is presently, in connection with general relativity, some kind of consensus to admit that photons have no mass. The mass should be in any case very small, cf. {\rm \cite{G}}. If this mass is under $10^{-69} kg$, it could be impossible to detect it, cf. \rm{\cite{Va}.}  \end{rmk}

\begin{rmk} As previously mentioned in Remark \ref{inc.}, Newton's framework is inconsistent with a constant velocity of photons. But it might happen that the velocity of light be the velocity of electromagnetic waves, while the velocity of photons could be much lower. It  is what happens with electricity, the velocity of electrons is infinitesimal compared to that of the electric signal propagation.  \end{rmk}

\begin{rmk} Light emanating orthogonally  from the surface of a black hole could be a unique example of a light ray stopping without any visible obstacle and coming backwards due to gravity, unfortunately it is probable that nobody will ever be able to observe such a phenomenon even in the distant future: approaching a high density black hole is probably difficult and dangerous, and low density black holes  may be out of reach for ever.  \end{rmk}

As a conclusion, we hope that the naive classical approach of the present paper will lead to further developments in more modern settings.


\begin{thebibliography}{99}

	
 	\bibitem {G} \textsc{A.S. Goldhaber and M.M. Nieto,}The Mass of the Photon, \newblock {Scientific American} Vol. 234, No. 5 (May 1976), pp. 86-97.
	
	\bibitem{Har} \textsc{E. R. Harrison,} \emph {Darkness at night, a riddle of the universe}, \newblock{Harvard University Press}, 1987, 293 pages. 	
		
	\bibitem{Mic}  \textsc{J. Michell,} On the Means of discovering the Distance, Magnitude, etc. of the Fixed Stars, in consequence of the Diminution of the velocity of their Light, in case such a Diminution should be found to take place in any of them, and such Data should be procured from Observations, as would be farther necessary for that Purpose, Philosophical transactions, lxxiv (1784), 35--57.
	
	\bibitem{Montes-T} \textsc{I.Montes \& I.Trujillo}, Intracluster light: a luminous tracer for dark matter in clusters of galaxies, 
	\newblock Monthly Notices of the Royal Astronomical Society, {\bf 482}, Issue 2, (January 2019),  2838--2851
	
	 \bibitem {S.S} \textsc {S. Schaffer,} John Michell and black holes, \newblock {Journal for the History of Astronomy} {\bf10}  (1979), 42--43. 	
	 
	 \bibitem {Sch} \textsc {K. Schwarzschild,} On the gravitational field of a sphere of incompressible fluid according to Einstein's theory, \newblock {Sitzungsber.Preuss.Akad.Wiss.Berlin (Math.Phys.)} 1916,  424--434

	 
	\bibitem {St} \textsc {W. M. Stuckey,} The observable universe inside a black hole,\newblock {American Journal of Physics} {\bf 62}, 788 (1994).

	 
	\bibitem{K Th} \textsc{Kip Thorne,} \emph{Black Holes and Time Warps : Einstein's Outrageous Legacy}, \newblock W W Norton \& Company, 1994, 619 p. (ISBN 0-393-31276-3)
	
	\bibitem{Va} \textsc{G. Vasseur,} {Limites sur la masse du photon}, \newblock CEA preprint, Saclay, 1996.
	
	\bibitem {W} \textsc{P.S. Wesson,}  Olbers's paradox and the spectral intensity of the extragalactic background light,
	\newblock Astrophysical Journal, Part 1 , vol. 367, Feb. 1, (1991)  399--406.
	
		
		
	\end{thebibliography}
\end{document}